\documentclass[12pt]{amsart}
\usepackage{amsmath}
\usepackage{geometry}
\usepackage{amsthm}
\usepackage{amsfonts}
\usepackage{comment}
\usepackage{amssymb}
\usepackage{eufrak}
\usepackage{comment}
\usepackage{hyperref}
\usepackage{mathtools}
\usepackage{color}
\numberwithin{equation}{section}
\usepackage{breqn}

\newtheorem{Theorem}{Theorem}[section]

\newtheorem{Lemma}[Theorem]{Lemma}
\newtheorem{Proposition}[Theorem]{Proposition}

\theoremstyle{definition}
\newtheorem*{Example}{Example}
\theoremstyle{remark}
\newtheorem*{remark}{Remark}

\newcommand{\leg}[2]{\left(\frac{#1}{#2}\right)}
\newcommand{\Tp}[1]{T\left(p^{#1}\right)}
\newcommand{\Tps}[2]{T_{#1}\left(p^{#2}\right)}

\usepackage{cite}

\title[$p$-adic Properties of Certain Half-Integral Weight Modular Forms]{$p$-adic Properties of Coefficients of Certain Half-Integral Weight Modular Forms}

\author{Lea Beneish}
\address{Department of Mathematics, Emory University, 400 Dowman Drive, W401, Atlanta, GA 30322}
\email{lea.beneish@emory.edu}

\author{Claire Frechette}
\address{Department of Mathematics, Brown University, Box 1917,
151 Thayer Street, Providence, RI 02912}
\email{\text{claire$\_$frechette@brown.edu}}

\begin{document}

\maketitle

\begin{abstract}
In this paper, we study the parallel cases of Zagier's and Folsom-Ono's grids of weakly holomorphic (resp. weakly holomorphic and mock modular) forms of weights 3/2 and 1/2, investigating their $p$-adic properties under the action of Hecke operators.
\end{abstract}

\section{Introduction and Statement of Results}


In \cite{ZAG}, Zagier builds two sets of interlocking weakly holomorphic modular forms of level 4 in Kohnen's plus space -- one set of weight 3/2 and one of weight 1/2 -- such that their coefficients not only form a grid, but also give the traces of singular moduli, the values of the $j$-function at CM-points. Zagier also uses these forms to reconstruct a theorem of Borcherds \cite{BO}, which enables the computation of minimal polynomials of these singular moduli. 

Inspired by Zagier's work, Duke and Jenkins explore in \cite{DJ} other half-integral weight weakly holomorphic modular forms in Kohnen's plus space of level 4. In \cite{BGK}, Bringmann, Guerzhoy, and Kane continue the study of these forms by investigating their $p$-adic properties, using a lifting procedure developed by Duke and Jenkins to link back to Zagier's original forms in order to give a $p$-adic relation between half-integral weight weakly holomorphic modular forms and classical half-integral weight holomorphic modular forms. However, the approach used by Bringmann, Guerzhoy, and Kane only works for half-integral weight forms of weight $k+\frac{1}{2}$ where $k \geq 2$. This raises the natural question of whether a similar result holds for forms lower weights, in particular, the original forms developed by Zagier.

Yet, Zagier's pair of sets of forms are not alone in their interconnectedness. In \cite{FO}, Folsom and Ono construct a startlingly similar grid: a set of weakly holomorphic modular forms of weight 3/2 which, when lined up term-by-term, form a set of mock modular forms of weight 1/2 in \cite{FO}, the first of which is essentially Ramanujan's third-order mock theta-function $$f(q) = 1 + \sum \limits_{n=1}^\infty \frac{q^{n^2}}{(1+q)^2(1+q^2)^2\cdots (1+q^n)^2}.$$ Guerzhoy follows up in \cite{GUER} by proving that all the coefficients of these forms are rational numbers with bounded denominator. Zwegers then strengthens this result in \cite{ZWEG} to prove that these coefficients are all integers.

Our goal in this paper is to obtain similar $p$-adic statements to those in \cite{BGK} for the parallel cases of Zagier's forms and Folsom-Ono's forms.
\begin{Theorem}\label{Hecke}
Let $v_p(\cdot)$ be the $p$-adic valuation, normalized such that $v_p(p) = 1$, and for $g$ a Fourier series with principal part $\sum_{\alpha} c_\alpha q^\alpha$ for a finite set of $\alpha$, define $w(g):= \max\{\lfloor \frac{v_p(\alpha)}{2}\rfloor\}$. Then, assuming the definitions in \S 2.1, 2.2, and 2.3, the following are true.
 \begin{enumerate} 
\item Let $p$ be an odd prime. Suppose g is a weight $\frac{3}{2}$ weakly holomorphic modular form in $M^{!,+}_{\frac{3}{2}}(4)$ with integer coefficients. Then for $n \geq w(g)$, we have $$g | \Tp{2n+4} - g|\Tp{2n} \equiv 0 \pmod{p^{n-w(g)}}.$$
\item Let $p\geq 5$ be a prime. Let G be a weight $\frac{3}{2}$ weakly holomorphic modular form in $M^{!,*}_{\frac{3}{2}}(144, \chi_{12})$ with integer coefficients. Then for $n \geq w(G)$, we have $$G | \Tps{12}{2n+4} - G|\Tps{12}{2n} \equiv 0 \pmod{p^{n-w(G)}}.$$
\end{enumerate}
\end{Theorem}

\begin{remark}
\end{remark}

\begin{enumerate}
 \item Theorem $1.1$ part (1) was proven independently in Ahlgren's ``Hecke relations for traces of singular moduli," by similar methods [\cite{Ahlgren}, Theorem $2$]. We were not aware of this result at the time we submitted our paper.
\item In \cite{ScottKim}, Ahlgren and Kim find similar relations to those in our Theorem $1.1$ for other grids, namely, grids that involve the function $\text{spt}(n)$, which counts the number of smallest parts in all partitions of $n$ and other smallest parts functions. They also find a Hecke relation for the mock theta function $f(q)$.
\end{enumerate}

\begin{Example} Let $p=3$ and take Zagier's form $g_4$ (see \S 2.2 for the construction) with principal part $q^{-4}$, $v_3(4)=0$, and let $n=1$.\\

We have $g_4|T_{p^2}$ and $g_4|T_{p^6}$ (computed modulo $3^9$, for convenience) as follows:\\
\begin{align*}
g_4|T_{p^2}&\equiv3q^{-36} + q^{-4} + 19675 + 19193q^3 + 6555q^4 + 13110q^7 + 
9665q^8 + 4197q^{11}\\
    & + 7517q^{12} + 8724q^{15} + 19665q^{16} + 13110q^{19} + 
O(q^{20}) \pmod{3^9}\\\\
g_4|T_{p^6}&\equiv27q^{-2916} + 9q^{-324} + 3q^{-36} + q^{-4} + 19603 + 19679q^3 
+ 19677q^4 +
     19671q^7 \\
     &+ 9665q^8 + 4197q^{11} + 19667q^{12} + 19659q^{15} + 
19665q^{16} +
     19671q^{19} + O(q^{20})\pmod{3^9}
\end{align*}
 \\
The  minimum $3$-adic valuation of the coefficients of $g_4|T_{p^6}- g_4|T_{p^2}$ (computed up 
to $O(q^{274}))$ is $2$, so $g_4|T_{p^6}- g_4|T_{p^2}\equiv 0 \pmod {3^2}$.
  
\end{Example}

In \cite{BGK}, Bringmann, Guerzhoy, and Kane explore the close $p$-adic relationship between the operators $U$ and $T$, exploiting the properties of the $T$-operator to yield a conclusion about the $U$-operator. To this end, we derive the following relation from the action of Hecke operators on our parallel sets of pairs of functions.

\begin{Theorem} \label{Today}  Assuming the same notation as in Theorem \ref{Hecke}, for $g_D$ (resp. $G_D$) the weight $\frac{3}{2}$ functions in 
Zagier's (resp. Folsom-Ono's) grid, the following are true. Let $j\in \mathbb{Z}_{\geq 0}$ with $p^2 \nmid j$. Then for$v,s \in \mathbb{Z}_{\geq 0}$.
\begin{enumerate}
\item Denote $g_D= \sum_d b(D,d)q^d$, then for $i\in \mathbb{Z}$ such that $\leg{-i}{p} = \leg{j}{p}$, we have 
\begin{align*}
b\left(p^{2v}j,p^{2v+2s}i\right) \equiv 0 \pmod{p^s}.
\end{align*}
\item Denote $G_D= \sum_d B(D,d)q^d$, then for $i \in \mathbb{Z}$ such that $\leg{-i}{p} = \leg{j}{p}$, we have 
\begin{align*}
B\left(p^{2v}j,p^{2v+2s}i\right) \equiv 0 \pmod{p^s}.
\end{align*}
\end{enumerate}
\end{Theorem}

This paper is organized as follows: in \S 2.1 we define notation and recall preliminaries. Then, in \S 2.2 and \S 2.3, we describe Zagier's and Folsom-Ono's half-integral weight weakly holomorphic modular forms in terms of their expansions, their actions under Hecke operators, and their duality properties. Finally, in \S 3, we prove Theorems $1.1$ and $1.2$. Throughout this paper, $p$ is taken to be a prime unless otherwise stated.

\section*{Acknowledgments}
The authors would like to thank Ken Ono for suggesting the topic and for advice and guidance throughout the process. We also would like to thank Michael Griffin, Michael Mertens, and Sarah Trebat-Leder for useful conversations. Both authors are also grateful to NSF for its support.

\section{Nuts and Bolts}

In this section, we define Zagier's weight $1/2$ and $3/2$ forms, describe the way the Hecke operators act on them, and state their duality properties. We also give the analogous description for those forms that make up the Folsom-Ono grid. First, we define some notation concerning Hecke operators and the spaces of these forms.
\subsection{Notation and Preliminaries}\emph{}

The congruence subgroup $\Gamma_0(N)$ is defined as
\[\Gamma_0(N) := \left\{\begin{pmatrix} a&b\\c&d\end{pmatrix} \in \mathrm{SL}_{2}(\mathbb{Z}) : c\equiv 0 \pmod{N}\right\}.\]

We call $f$ a \emph{weakly holomorphic modular form} of weight $k+\frac{1}{2}$ and level $4N$ and Nebentypus $\chi$ if it is a holomorphic function on the upper half-plane $\mathfrak{H}$ that satisfies
\[f\left(\frac{a\tau+b}{c\tau+d}\right)=\chi(d)\leg{c}{d}^{2k+1}\epsilon_d^{-2k-1}(c\tau+d)^{k+\frac{1}{2}}f(\tau) \quad \]
 \text{for all $\begin{pmatrix}
a & b  \\
c & d \end{pmatrix} \in \Gamma_0(4N)$}, where, $\epsilon_d$ is either $1$ or $i$ depending on whether $d\equiv 1\pmod 4$ or $d\equiv 3\pmod 4$,
and its poles, if any, are supported at the cusps (see \cite{Ken}).

 A \emph{cusp form} is a modular form that vanishes at all cusps. Similarly, a \emph{weakly holomorphic cusp form} is a weakly holomorphic modular form which has zero constant term at all cusps \cite{Ken}. 
 
We denote the space of holomorphic (resp. weakly holomorphic) modular forms of weight $k+\frac{1}{2}$ and level $N$, with Nebentypus $\chi$ by $M_{k+\frac{1}{2}}(N,\chi)$ (resp. $M_{k+\frac{1}{2}}^!(N,\chi)$). We write $M_{k+\frac{1}{2}}^{!,+}(N, \chi)$ to emphasize when the forms in $M_{k+\frac{1}{2}}^!(N,\chi)$ are in Kohnen's plus space. Similarly, we write $S_{k+\frac{1}{2}}(N,\chi)$ (resp. $S_{k+\frac{1}{2}}^!(N,\chi)$) for the space of holomorphic (resp. weakly holomorphic) cusp forms of weight $k+\frac{1}{2}$ and level $N$ with Nebentypus $\chi$ (and for weakly holomorphic cusp forms in the plus space, $S_{k+\frac{1}{2}}^{!,+}(N, \chi)$). 

Let $M^{!,*}_{k+\frac{1}{2}}(N,\chi)$ signify the space spanned by Folsom-Ono's forms of weight $k + \frac{1}{2}$. That is, we denote by $M^{!,*}_{k}(N,\chi)$ the space of weakly holomorphic modular forms of weight $k$, level $N$, with Nebentypus $\chi$ whose coefficients are supported on exponents that are congruent to $(-1)^{k+1} \pmod{24}$. 

Furthermore, following the notation in \cite{BRO}, we define a \emph{harmonic Maass form} on $\Gamma_0(4N)$ of weight $k+\frac{1}{2}$ as a smooth function $H:\mathfrak{H} \to \mathbb{C}$ such that
\begin{enumerate}
\item For all matrices in $\Gamma_0(4N)$,  \[H\left(\frac{a\tau+b}{c\tau+d}\right)=\leg{c}{d}^{2k+1}\epsilon_d^{-2k-1}(c\tau+d)^{k+\frac{1}{2}}H(\tau) \quad \]
\item $\Delta_{k+\frac{1}{2}}H = 0$ where $\Delta_k$ is the weight $k$ hyperbolic Laplacian, defined as $$\Delta_k := -y^2\left(\frac{\partial^2}{\partial x^2} + \frac{\partial^2}{\partial y^2}\right)  + iky\left(\frac{\partial}{\partial x} + i\frac{\partial}{\partial y} \right).$$ for $\tau=x+iy\in \mathfrak{H}$ with $x,y\in\mathbb{R}.$ 

\item There is a polynomial $P_H=\sum_{n\geq 0} c^{+}(n)q^n\in \mathbb{C}[q^{-1}]$ such that $$H(\tau)-P_H(\tau)=O(e^{-\epsilon y})$$ as $y\to +\infty$ for some $\epsilon>0$, and analogous conditions are required at all cusps.
\end{enumerate}
If $H$ is a harmonic Maass form, there is a canonical splitting of $H$ into  $$H = H^+ + H^-$$
where $H^+$ is the holomorphic part of $H$, called a \emph{mock-modular form}, and $H^-$ is the non-holomorphic part of $H$ \cite{unearthing}.

We also recall the classical cuspidal Poincar\'e series and the Maass-Poincar\'e series as in \cite{unearthing}: a general Poincar\'e series of weight $k$ for $\Gamma_0(N)$ is given by
\[\mathbb{P}(m,k,N,\varphi_m;\tau):=\sum\limits_{\gamma\in\Gamma_\infty\setminus \Gamma_0(N)}(\varphi_m^*|_k\gamma)(\tau),\]
where $m$ is an integer, $\Gamma_\infty:=\left\{\pm\left(\begin{smallmatrix} 1 & n \\ 0 & 1 \end{smallmatrix}\right)\: :\: n\in\mathbb{Z}\right\}$ is the subgroup of translations in $\Gamma_0(N)$, and $\varphi_m^*(\tau):=\varphi_m(y)e^{2\pi imx}$ for a function $\varphi_m:\mathbb{R}_{>0}\rightarrow\mathbb{C}$ which is $O(y^A)$ as $y\rightarrow 0$ for some $A\in\mathbb{R}$. We
distinguish two special cases ($m>0$),
\begin{align*}
P(m,k,N;\tau)&:=\mathbb{P}(m,k,N,e^{-my};\tau)\\
Q(-m,k,N;\tau)&:=\mathbb{P}(-m,2-k,N,\mathcal{M}_{1-\frac k2}(-4\pi my);\tau),
\end{align*}
where $\mathcal{M}_s(y)$ is defined in terms of the $M$-Whittaker function. We often refer to $Q(-m,k,N;\tau)$ as a \emph{Maass-Poincar\'e series}. It is well known that the Fourier expansions of the cuspidal Poincar\'e series
are given by infinite sums of Kloosterman sums weighted by
$J$-Bessel functions \cite{unearthing}.

We also define the Petersson inner product for forms $f(\tau)\in M_k(N)$ and $g(\tau) \in M_k^!(N)$ (note that it also exists for $f(\tau)\in M_k(N)$ and $g(\tau)$ a harmonic Maass form), denoted $\langle f, g \rangle$, as the constant term in the expansion at $s=0$ of the meromorphic continuation in $s$ of the function
$$\frac{1}{[SL_2(\mathbb{Z}): \Gamma_0(N)]}\lim\limits_{T\to\infty}\int_{\mathcal{F}_T(N)} f(\tau)\overline{g(\tau)}y^{k-s-2}dxdy$$ where $$\mathcal{F}_T(N):=\bigcup_{\gamma \in \Gamma_0(N)\backslash SL_2(\mathbb{Z})}\gamma\mathcal{F_T}(SL_2(\mathbb{Z})$$
and $$\mathcal{F_T}(SL_2(\mathbb{Z})=\{\tau \in \mathfrak{H}|\hspace{1mm} |x|\leq \frac{1}{2}, |\tau|\geq 1, \text{and }  y\leq T\}.$$
We define the differential operator by $D:=\frac{1}{2\pi i}\frac{d}{dz}$ and let $\mathcal{H}_k(N)$ denote the space of harmonic Maass forms of weight $k$ and level $N$.
\begin{Lemma} (Theorem $7.8$ in\cite{unearthing}) If $2\leq k\in \mathbb{Z}$ the image of the map $$D^{k-1}: \mathcal{H}_{2-k}(N)\to M_k^!(N)$$ consists of forms in $M_k^!(N)$ which are orthogonal to cusp forms with respect to the regularized inner product, which also have constant term zero at all cusps of $\Gamma_0(N)$.
\end{Lemma}

Let $m$ be an integer and let $q:= e^{2\pi i \tau}.$ Following \cite{BGK}, we also recall, the standard $U(m)$ and $V(m)$ operators by their action on $q$-series $\sum a(n)q^n$ as follows,

{\begin{align*}
\left(\sum a(n)q^n\right) | U\left(m\right) &:= \sum a\left(mn\right)q^n,\\
\left(\sum a(n)q^n\right) | V\left(m\right) &:= \sum a(n)q^{mn}.
\end{align*}

Let $\chi(t,k),$ where $t$ and $k$ are integers, be the twisting operator, defined by
\begin{align*}
\left(\sum a(n)q^n\right)  \otimes \chi(t,k) &:= \sum \leg{(-1)^k\cdot n}{t}a(n)q^n, 
\end{align*} 
where $\leg{\cdot}{\cdot }$ is the standard Kronecker symbol.
Note that all three of these operators preserve modularity, although they may change the level (see section 3.2 of \cite{Ken}).

The weight $k+\frac{1}{2}$ Hecke operators are defined for $f$ of level $N$ and Nebentypus $\chi_s=\leg{s}{\cdot}$, 
where $p$ is a prime and $p \nmid N$, by 
\begin{align}
\label{Top} f|T_{s}\left(p^2\right) := f|U\left(p^2\right) + p^{k-1}\leg{s}{p}f\otimes \chi(p,k) + p^{2k-1}f|V\left(p^2\right),
\end{align} 
We then define $\Tps{s}{2m}$ for $m \geq 1$ recursively by
\begin{align}
\label{Toprec}\Tps{s}{2m} &:= \Tps{s}{2m-2}\Tps{s}{2} - p^{2k-1}\Tps{s}{2m-4}.
\end{align}
Furthermore, this can be extended to the formula
\begin{align}\label{multiH}
\Tps{s}{2m}\Tps{s}{2n} &= \sum_{t = 0}^m p^t \cdot \Tps{s}{2n+ 2m -4t} , \text{ for $m \leq n$}. 
\end{align}
However, note that when we use the weight $\frac{1}{2}$ Hecke operator $\Tp{2}$, we use the normalized version $p\cdot \Tp{2}$ as in \S 6 of \cite{ZAG}. Also, in the case where the Nebentypus is the trivial character, as it is in Section 2.2, we suppress the $s$ in the notation of the Hecke operator.

\subsection{Zagier's Half-Integral Weight Modular Forms}

We define Zagier's forms $f_d$ and $g_D$ of weight $\frac{1}{2}$ and $\frac{3}{2}$, respectively. Following Zagier's notation, we label the forms $f_d$ and $g_D$ each according to the single term in their principal part, where the subscripts $d$ and $D$ are such that $f_d = q^{-d} + \sum\limits_{D> 0} a(D,d)q^D$ and $g_D = q^{-D} + \sum\limits_{d \geq 0} b(D,d)q^d$. These forms are weakly holomorphic modular forms with integral coefficients on $\Gamma_0(4)$ in Kohnen's plus space which form a basis for $M_{\frac{1}{2}}^{!,+}(\Gamma_0(4))$ and $M_{\frac{3}{2}}^{!,+}(\Gamma_0(4))$, respectively (see \cite{ZAG}). First we define the following:
$$\theta(\tau):=\sum_{n=-{\infty}}^{\infty}q^{n^2}, \hspace{5mm} \theta_1(\tau):=\sum_{n=-\infty}^{\infty} (-1)^nq^{n^2}, \hspace{5mm} g(\tau):=\theta_1(\tau)\frac{E_4(4\tau)}{\eta(4\tau)^6}.$$

The basis of $M_{\frac{1}{2}}^{!,+}(\Gamma_0(4))$, the $f_d$'s, can be constructed as follows: $f_0:=\theta(\tau)$ and $f_3$ can be obtained from $[\theta(\tau), E_{10}(4\tau)]/\Delta(4\tau)$, which is a linear combination of $f_0$ and $f_3$. The rest of the $f_d$ can be computed by multiplying $f_{d-4}$ by $j(4\tau)$ and subtracting off multiples of the previously computed $f_j$ $(0\leq j <d)$ so that $f_d$ has the form $f_d = q^{-d}+O(q)$.
\begin{align*}
f_0=&1+2q+2q^4+2q^9+2q^{16}+O(q^{25})\\
f_3=&q^{-3}-248q+26752q^4-85995q^5+1707264q^8-4096248q^9+O(q^{12})\\
f_4=&q^{-4}+492q+143376q^4+565760q^5+18473000q^8+51180012q^9+O(q^{12})\\
\vdots
\end{align*}
The construction for the basis of $M_{\frac{3}{2}}^{!,+}(\Gamma_0(4))$, the $g_D$'s, is the same as that for the $f_d$'s, except we start at $g_1:=g$ and construct $g_4$ as we constructed $f_3$ (with $g$ in place of $\theta$). In particular, the space of holomorphic modular forms of weight $3/2$ on $\Gamma_0(4)$ in the plus space is empty, and so there are no $g_D$'s without a pole.

\begin{align*}
g_1=&q^{-1}-2+248q^3-492q^4+4119q^7-7256q^8+33512q^{11}-53008q^{12}+O(q^{15})\\
g_4=&q^{-4}-2-26752q^3-143376q^4-8288256q^7- 26124256q^8+O(q^{11})\\
g_5=&q^{-5}+ 0+85995q^3-565760q^4+52756480q^7-190356480q^8+O(q^{11})\\
\vdots
\end{align*}

A brief scan will show that the coefficients of $g_D$ appear as the negatives of the $D$-th coefficients of $f_d$. In fact, this duality is true in general, as described in the following theorem from \cite{ZAG}.

\begin{Theorem}\label{AB} Let $f_d = q^{-d} + \sum\limits_{D> 0} a(D,d)q^D$ and $g_D = q^{-D} + \sum\limits_{d \geq 0} b(D,d)q^d$, then \begin{enumerate}

\item For all $m \geq 0$, let $g_D | T\left(m^2\right) = \sum_d b_m(D,d)q^d$ and $f_d | T\left(m^2\right) = \sum_D a_m(D,d)q^D$. For all $D$ and $d$ such that $D \equiv 0,1 \pmod{4}$ and $d \equiv 0,3\pmod{4}$, 
\begin{align}
\label{aandb} a_m(D,d) = -b_m(D,d).
\end{align}
\item For any $m$, the following recursion holds $$a_m(1,d)=\sum\limits_{n|m}n\cdot a\left(n^2,d\right).$$
\end{enumerate}
\end{Theorem}

The following proposition describes the action of the Hecke operators $\Tp{2n}$ on these $g_D$, extending the $\Tp{2}$ case done in \S 6 of \cite{ZAG}.
\begin{Proposition} \label{HeckegD} Let $g_D$ be one of Zagier's weight 3/2 weakly holomorphic modular forms on $\Gamma_0(4)$. We write $D = p^{2v}\cdot j,$ where $p^2 \nmid j$ and $v \geq 0$. Then,  for all $0 \leq n < v$, we have
$$g_D | \Tp{2n} = \sum_{t = 0}^{n} p^{t}\cdot g_{p^{2v -2n+ 4t}\cdot j},$$
and for all $n \geq v$, we have
\begin{align}
\label{eqgD} g_D | \Tp{2n} = \sum_{t = 0}^{n-v} \leg{j}{p}^{n-v-t}p^t\cdot g_{p^{2t}j} + \sum_{t = 1}^{v} p^{n-v+t}\cdot g_{p^{2n -2v+ 4t}\cdot j}.
\end{align}
\end{Proposition}
\vspace{0.5cm}
\begin{proof}
We first consider $g_D$ where $D = j, p^2\nmid j$. Recall that $g_j$ looks like
$$g_j = q^{-j} + O(1).$$
Acting on $g_j$ with the $p^{2n}$-th Hecke operator, we look just at the principal part. First, the base case of $\Tp{2}$. From the definition in \eqref{Top}, 
\begin{align*}
g_j | \Tp{2} &= g_j | U\left(p^2\right) + g_j \otimes \chi(p,1) + p\cdot g_j | V\left(p^2\right)\\
&= p\cdot q^{-p^2j}+ \leg{j}{p}q^{-j} + O(1)\\
&= pg_{p^2j} + \leg{j}{p}g_j,
\end{align*}
where the last line follows from the fact that Hecke operators do not change the level of the form, so $g_j | \Tp{2}$ will remain in the space spanned by Zagier's forms. Since these form a basis, any form in this space is determined entirely by its principal part.
Proceeding inductively, assume that
\begin{align}\label{explicit} g_j | \Tp{2\ell} = \sum_{t = 0}^\ell \leg{j}{p}^{\ell-t}p^t\cdot g_{p^{2t}j}\end{align} holds for all $l\leq n$, for some $n\geq 0$.
Then, 
{\allowdisplaybreaks\begin{align}
\notag g_j | \Tp{2n+2} &= g_j | \Tp{2n}\Tp{2} - p\cdot g_j | \Tp{2n - 2}\\
\label{bbb} &= \left(\sum_{t = 0}^n \leg{j}{p}^{n-t}p^t\cdot g_{p^{2t}j}\right) | \Tp{2} -\sum_{t = 0}^{n-1} \leg{j}{p}^{n-t-1}p^{t+ 1}\cdot g_{p^{2t}j}
\end{align}
by applying \eqref{multiH}, so we have that \eqref{bbb} condenses to
\begin{align*}
&= \sum_{t = 1}^n \leg{j}{p}^{n-t}p^t\cdot g_{p^{2t-2}\cdot j} +\leg{j}{p}^{n+1}g_j 
\\
&= \sum_{t = 0}^{n+1} \leg{j}{p}^{n-t+1}p^{t}\cdot g_{p^{2t}j}.
\end{align*}}
Note that this proves the proposition for $m = 0$ since then the second sum in the statement will be empty. \\

We now consider $D = p^{2v}j, v \geq 1,$ where $p^2 \nmid j$. By \eqref{explicit}, we can rewrite $g_D$ as 
\[ g_D = \frac{g_j | \Tp{2v} - \leg{j}{p}g_j | \Tp{2v -2}}{p^v}.\]
Then, for $n \geq v$,
{\allowdisplaybreaks\begin{align*}
g_D | \Tp{2n} &= p^{-v} \left( g_j | \Tp{2v}\Tp{2n} - \leg{j}{p}g_j | \Tp{2v -2}\Tp{2n}\right)\\
&= p^{-v} g_j | \left(\sum_{t = 0}^v p^t \cdot \Tp{2n+ 2v -4t} - \sum_{t = 0}^{v-1}\leg{j}{p} p^t \cdot \Tp{2n+ 2v -4t-2}\right)\\
\intertext{again by applying \eqref{multiH}, so we have}
&= g_j | \Tp{2n-2v} + p^{-v} \sum_{t = 1}^{v} p^{n+t}\cdot g_{p^{2n-2v + 4t }\cdot j}\\
&= \sum_{t = 0}^{n-v} \leg{j}{p}^{n-v-t}p^t\cdot g_{p^{2t}j}+ \sum_{t = 1}^{v} p^{n-v+t}\cdot g_{p^{2n -2v+ 4t}\cdot j}.
\end{align*}}
The proof for $0 \leq n < v$ follows similarly, using the relation for $\Tp{2n}\Tp{2v}$ and summing over $t$ from $0$ to $n$, then reordering the sum.
\end{proof}
\vspace{0.5cm}
\begin{Proposition}\label{Heckefd}
Let $f_d$ be one of Zagier's weight 1/2 forms. We write $d = p^{2u}\cdot i,$ where $p^2 \nmid i$ and $u \geq 0$. Then, for all $0\leq n < u$, we have
$$f_d | \Tp{2n} = \sum_{t = 0}^{n}p^{n-t} f_{p^{2u-2n +4t}\cdot i},$$
and for $n \geq u$, we have
\begin{align}
\label{eqfd} f_d | \Tp{2n} = \sum_{t = 0}^{n-u} \leg{-i}{p}^{n-u-t}p^uf_{p^{2t}i} + \sum_{t = 1}^{u}p^{u-t} f_{p^{2n-2u +4t}\cdot i}.
\end{align}
\end{Proposition}
\begin{proof} The proof of this proposition follows in the same way as that of Proposition \ref{HeckegD}, with the exception that the weight $\frac{1}{2}$ Hecke operator is normalized as in \S 6 of \cite{ZAG}, so we instead have $$f_d | \Tp{2} = p\cdot f_d | U\left(p^2\right) + f_d \otimes \chi(p,1) + f_d | V\left(p^2\right).$$
\end{proof}
\subsection{Folsom-Ono Grid of Half-Integral Weight Forms}

In $3.5$ of \cite{FO}, Folsom and Ono introduce a grid of two sets of forms analogous to those of Zagier. Changing notation to avoid confusion with Zagier's $f_d,g_D$, let $\tilde{f}_m,\tilde{g}_m$ be the forms defined in Folsom and Ono's paper, using their notation with $m\in\mathbb{Z}$ as an index. We then define the sets $F_d$ and $G_D$ as follows, where $d,D$ respectively each denote the integral power from the only term in the principal part:
\begin{align*}
F_{24m - 23}(q) &= q^{-1}\tilde{f}_m\left(q^{24}\right) \hspace{3 cm} G_{24m - 1}(q) = -q\tilde{g}_m\left(q^{24}\right)\\
 &=-q^{-24m + 23} + O\left(q^{23}\right), \hspace{3.2 cm} = q^{-24m+1} + O(q).\\
\end{align*}
One set consists of weight 1/2 mock modular forms $F_d$. Note that $F_1(q)=q^{-1}\tilde{f}_1(q^{24})$, where $\tilde{f}_1(q)$ is Ramanujan's third order mock-theta function 
$$\tilde{f}_1(q) =f(q)= 1+ \sum_{n=1}^\infty \frac{q^{n^2}}{(1+q)^2(1+ q^2)^2\cdots(1+q^n)^2}.$$ 
The $F_d$'s have $q$-expansions as follows: $$F_d(q) : = -q^{-d} + \sum_{n=1}^\infty A(D,d)q^D.$$

Then, we have, in one direction,
\begin{align*}
F_1(q) &= -q^{-1} + q^{23} - 2q^{47} + 3q^{71} - \cdots\\
F_{25}(q) &= -q^{-25} - 263 q^{23} + 2781q^{47} - 17960q^{71} + \cdots\\
F_{49}(q)&= -q^{-49} + 3400q^{23} - 102060q^{47} + \cdots\\
F_{73}(q) &= -q^{-73} -23374q^{23}+\cdots\\
\end{align*}
Folsom and Ono prove that by grouping coefficients of the grid in columns, as in Zagier's case above, \S 3.3  of \cite{FO} formally defines a set of $q$-series $G_D$, which look like
\begin{align*}
G_D(q) := q^{-D} + \sum_{n=1}^\infty B(D,d)q^d.
\end{align*}
For example, we have
\begin{align*}
G_{23}(q) &= q^{-23} -q + 263 q^{25} -3400 q^{49} + 23374 q^{73} + \cdots\\
G_{47}(q)&= q^{-47} +2q-2781q^{25} + 102060q^{49} + \cdots\\
G_{71}(q) &= q^{-71} -3q+ 17960q^{25}+\cdots\\
\end{align*}
These $G_D$ are then proven to be weight 3/2 weakly holomorphic modular forms; more precisely, they live in $M^!_{\frac{3}{2}}(\Gamma_0(144),\chi_{12})$. In analogy with Theorem \ref{AB}, Folsom and Ono prove the following
\begin{Theorem}\label{FOAB} For $F_d$ and $G_D$ as above, we have the relation 
\begin{align}
\label{AandB} A(D,d)= -B(D,d).
\end{align}
\end{Theorem}

As noted in the introduction, due to work of Guerzhoy in \cite{GUER} and Zwegers in \cite{ZWEG}, these coefficients are integral.

Noting the similarity to the corresponding of the coefficients of the Zagier forms, we ask, as Zagier did, if the action of the Hecke operators preserves this duality. However, we first need to check that we can determine the action of the Hecke operators on $G_D$ and $F_d$ solely by looking at the principal part.

\begin{Lemma}\label{BasisGD}
The space spanned by Folson-Ono's forms $G_D \in S^!_{\frac{3}{2}}(\Gamma_0(144),\chi_{12})$,  supported on powers of $q$ which are congruent to $1 \pmod{24}$, is invariant under the action of Hecke operators $\Tps{12}{2n}$, where $p \nmid 144$, so the action of $\Tps{12}{2n}$ on any form in that space is determined entirely by its principal part.
\end{Lemma}
\begin{proof} First, let $G$ be a form in the aforementioned space such that $G = G_{D_0}| \Tps{12}{2n}$ for some $D_0$. Note that $D_0 = 24r + 1$ for some $r \in \mathbb{Z}, r \geq 1$, and that $G_{D_0}$ is supported only on coefficients equivalent to $1 \pmod{24}$.

We first reexamine the Hecke operators under this Nebentypus to verify that $G$ is supported on the correct arithmetic progression of exponents. We recall from \eqref{Top} that, $$\Tps{12}{2} = U\left(p^2\right) + \leg{12}{p}\chi(p,1) + pV\left(p^2\right).$$ It is a well-known fact that for all $p \geq 5$, $p$ prime, $p^2 \equiv 1 \pmod{24}$, so we know that the space of $G_D$ where $D \equiv 1 \pmod{24}$ is Hecke-invariant, since any action of Hecke operators will change $q$-powers of $n$ to a linear combination of $q$-powers $p^{2t}n \equiv n \pmod{24}$. Thus, the resulting forms under Hecke transformations will also be supported on powers of $q$ which are equivalent to $1 \pmod {24}.$\\

Let us denote the principal part of $G$ as $\sum\limits_{\substack{n<0\\n \equiv 1 \pmod{24}}} c_nq_n$. We can then construct a sum of $G_D$'s to have the same principal part as $G$: $$G^\prime = \sum\limits_{\substack{n<0\\n \equiv 1 \pmod{24}}}c_n G_n.$$
Since the $G_D$ are multiples of Poincar\'{e} series (see section $3.5$ in \cite{FO} for the proof) constructed to have principal part zero at all other cusps but infinity, $G^\prime$ will have principal part zero at all other cusps as well. Furthermore, since the Hecke operator $T_{12}$ cannot introduce a principal part, $G$ will also have principal part zero at all other cusps. We also know, from \cite{FO}, that the $G_D$ have constant term zero at all cusps, so since $T_{12}$ also cannot introduce a constant term,  $G$ also has constant term zero at all cusps. Therefore, we have that $G - G^\prime \in S_{\frac{3}{2}}(\Gamma_0(144), \chi_{12})$, the space of holomorphic cusp forms with the same level and Nebentypus. If this space were empty, we would know that $G-G^\prime$ would be zero, and as a result that $G$ would be uniquely determined by its principal part. However, this space has dimension 2, and is generated by
\begin{align*}
B_1 &= q-5q^{25}+7q^{49}-11q^{121}+13q^{169}+ \cdots\\
B_2 &= q^4-2q^{16}+4q^{64}-5q^{100}+\cdots
\end{align*}
Thus, $G - G^\prime  = a_1B_1,$ since $B_2$ is supported on the wrong progression. However, we also claim that $a_1 = 0$, once again using the fact that the $G_D$ are Poincar\'{e} series. Namely, we know by \cite{ALF}  that they are orthogonal to all cusp forms, i.e. $\langle G_D, g\rangle = 0$ for a cusp form $g$, where $\langle\cdot,\cdot\rangle$ denotes the regularized Petersson inner product (see \S2.1 and \cite{BOR,BRO,unearthing}). 
We know that for any $f\in M^!_k$ and $g\in S_k$, $$\langle f|\Tps{12}{2n},g\rangle=\chi(n)\cdot \langle f,g|\Tps{12}{2n}\rangle.$$ Then, since $\langle G_D,g\rangle = 0$ for any cusp form $g$, 
$$\langle G_D | \Tps{12}{2n},g\rangle = \chi(n) \langle G_D, g|\Tps{12}{2n}\rangle = 0,$$ for $n$ coprime to the level, since the space of cusp forms is Hecke-invariant under $\Tps{12}{2n}$.
From here, note that the Petersson inner product is linear in the first argument and some linear algebra will show that, in fact, $a_1 = 0$.
We can extend this argument to $G = \left(\sum_D G_D\right) | \Tps{12}{2n}$ by noting that $$G = \left(\sum_D G_D\right) | \Tps{12}{2n} =  \sum_D \left( G_D | \Tps{12}{2n}\right),$$ so if we again form $G^\prime$ a sum of $G_n$'s with the same principal part as $G$, then $$\langle G-G^\prime,B_1\rangle = \langle G, B_1\rangle = \sum_D \langle G_D | \Tps{12}{2n}, B_1\rangle = 0.$$ Thus, the action of Hecke operators $\Tps{12}{2n}$ on the space spanned by Folsom-Ono's $G_D$'s is entirely determined by the principal part.
\end{proof}

\begin{Lemma}\label{BasisFd}
The space spanned by Folsom-Ono's forms $F_d \in S^!_{\frac{1}{2}}(\Gamma_0(144),\chi_{12})$,  supported on powers of $q$ which are congruent to $-1 \pmod{24}$, is invariant under the action of the normalized Hecke operators $\Tps{12}{2n}$, where $p \nmid 144$, so the action of $\Tps{12}{2n}$ on any form in that space is determined entirely by its principal part.
\end{Lemma}

\begin{proof}
As in the case of the weight $\frac{1}{2}$ Zagier forms, we use the normalized Hecke operators $\Tps{12}{2}  = pT_{\frac{1}{2},{12}}\left(p^{2}\right)$. So by \eqref{Top}, $$\Tps{12}{2} = pU\left(p^2\right) + \leg{12}{p}\chi(p,0) + V\left(p^2\right).$$
As in Lemma \ref{BasisGD}, for $p \geq 5$, $p$ prime, we know that and $p^2 \equiv 1 \pmod{24}$, so the action of the Hecke operator will preserve the arithmetic progression of the exponents, which are all supported on $q$-powers equivalent to $-1 \pmod{24}$.
Then, since Hecke operators are additively distributive, it suffices to consider $F$ a form in the aforementioned space such that $F = F_{d_0} |\Tps{12}{2n}$. Using the normalization, $F$ will have principal part with integral coefficients supported on $q$-powers $\equiv -1 \pmod{24}$, so we can construct a linear combination $F^\prime  = \sum_\alpha c_\alpha F_\alpha$ to have the same principal part as $F$. Now, we take the harmonic Maass form $H=H_{d_0}|\Tps{12}{2n}$ associated to $F$ and take the harmonic Maass form $H^\prime$ associated to $F^\prime$. We are interested in $F - F^\prime$ so we first consider $H-H^\prime$, a weight $\frac{1}{2}$ harmonic Maass form. In \cite{GUER}, Guerzhoy proves that such a harmonic Maass form will have a shadow which is a cusp form. However, then, $H - H^\prime$ has zero principal part at all cusps and thus must be $0$, since every harmonic Maass form with nonzero shadow must have a pole at some cusp, which follows from Theorem 3.6 of \cite{BF}. Then, since Hecke operators respect the holomorphic and non-holomorphic parts of harmonic Maass forms, $F-F^\prime = 0$, so the action of Hecke operators on individual $F_{d}$ is entirely determined by principal part.

\end{proof}

Denote the coefficients under the action of the Hecke operators as $F_d |_{\frac{1}{2}} T\left(m^2\right) = \sum_D A_m(D,d)$ and $G_D |_{\frac{3}{2}} T\left(m^2\right) = \sum_d B_m(D,d)$. 
\begin{Theorem}\label{FOABM}
For all $m, D,d \in \mathbb{Z}$, 
\begin{align}
\label{AandBm}
A_m(D,d) = -B_m(D,d).
\end{align}
\end{Theorem}

\begin{proof} It suffices to prove the statement for $m = p$ prime. By the definition of the $3/2$ weight Hecke operator in \eqref{Top}, $$B_p(D,d) = B(D,p^2d) + \leg{12}{p}\leg{-d}{p}B(D,d) + pB(D,p^{-2}d).$$
If we consider $d < 0$, we have from the definition of the $G_D$ that $$B(D,d) = \delta_{-D,d},$$ where $\delta_{i,j}$ is the Kronecker $\delta-$function. Thus,
\begin{align*}
B_p(D,d) &= \delta_{-D,p^2d} + \leg{12}{p}\leg{-d}{p}\delta_{-D,d} + p\delta_{-D,p^{-2}d}\\
&= \delta_{-p^{-2}D,d} + \leg{12}{p}\leg{D}{p}\delta_{-D,d} + p\delta_{-p^2D,d}.
\end{align*}
Since the action of the Hecke operators on the $G_D$ is determined entirely by their principal parts, by Lemma \ref{BasisGD},
$$G_D | \Tps{12}{2} = G_{p^{-2}D} + \leg{12}{p}\leg{D}{p}G_D + pG_{p^2D}.$$
Therefore, for all $ D,d \in\mathbb{Z}$, 
\begin{align*}
B_p(D,d) &= B(p^{-2}D,d) + \leg{12}{p}\leg{D}{p} B(D,d) + pB(p^2D,d)\\
&= -A(p^{-2}D,d) - \leg{12}{p}\leg{D}{p}A(D,d) - pA(p^2D,d)\\
\intertext{by \eqref{AandB}, so we have that}
B_p(D,d) &= -A_p(D,d)
\end{align*}
by the definition of the Hecke operators for weight $1/2$ in \eqref{Top}.
\end{proof}
Moreover, since these results appear so similar to those in the previous section, it is natural to ask the analogous questions answered in Propositions \ref{HeckegD} and \ref{Heckefd}. 

\begin{Proposition} \label{HeckeGD} Let $G_D$ be one of Folsom-Ono's weight 3/2 weakly holomorphic cusp forms on $\Gamma_0(144)$ with Nebentypus $\chi_{12}$. 
We write $D = p^{2v}\cdot j, v \geq 0,$ where $p^2 \nmid j$. 
Then,  for all $0 \leq n < v$, we have
$$G_D | \Tps{12}{2n}= \sum_{t = 0}^{n} p^{t}\cdot G_{p^{2v -2n+ 4t}\cdot j},$$
and for all $n \geq v$, we have
\begin{align}
\label{eqGD}G_D | \Tps{12}{2n}= \sum_{t = 0}^{n-v} \leg{12j}{p}^{n-v-t}p^t\cdot G_{p^{2t}j} + \sum_{t = 1}^{v} p^{n-v+t}\cdot G_{p^{2n -2v+ 4t}\cdot j}.
\end{align}
\end{Proposition}
\begin{Proposition}\label{HeckeFd}
Let $F_d$ be one of Folsom-Ono's weight 1/2 mock modular forms on $\Gamma_0(144)$ with Nebentypus $\chi_{12}$. 
We write $d = p^{2u}\cdot i,$ where $p^2 \nmid i$ and $u \geq 0$. Then, for all $0\leq n < u$,
$$F_d | \Tps{12}{2n} = \sum_{t = 0}^{n}p^{n-t} F_{p^{2n-2u +4t}\cdot i},$$
and for $n \geq u$,
\begin{align}
\label{eqFd} F_d | \Tps{12}{2n} = \sum_{t = 0}^{n-u} \leg{-12i}{p}^{n-u-t}p^uf_{p^{2t}i} + \sum_{t = 1}^{u}p^{u-t} F_{p^{2n-2u +4t}\cdot i}.
\end{align}
\end{Proposition}

\begin{proof} [Proof of Proposition \ref{HeckeGD}]
Let $G_D$ be one of Folsom-Ono's forms in \cite{FO}. Given Lemma \ref{BasisGD}, the proof of Proposition \ref{HeckeGD} follows similarly to that of Proposition \ref{HeckegD}, with the exception that a factor of $\leg{12}{p}$ appears along with each factor of $\leg{j}{p}$.
\end{proof}

\noindent The proof of Proposition \ref{HeckeFd} is similarly analogous to that of Proposition \ref{Heckefd}.

\section{Proofs of the Main Results}

\subsection{Proof of Theorem \ref{Hecke}}
We first prove two Propositions needed to prove Theorem \ref{Hecke}.

\begin{Proposition} \label {gD} Let $g_D$ be one of Zagier's forms. We write $D = p^{2v}\cdot j$, where $v\geq 0$ and $p^2 \nmid j$. For $n\geq v$, we have
\begin{align}
\label{eqgDmod} g_D | \Tp{2n + 2} - \leg{j}{p}\cdot g_D | \Tp{2n} \equiv 0 \pmod{p^{n-v+1}}.
\end{align}
\end{Proposition}
\begin{proof}Suppose that $D = j, p^2 \nmid j$. Then, if we consider the difference of successive Hecke operators, by \eqref{eqgD} of Proposition \ref{HeckegD}:
\begin{align*}
g_j | \Tp{2n + 2} - \leg{j}{p} g_j | \Tp{2n} &= \sum_{t = 0}^{n+1} \leg{j}{p}^{n-t+1}p^{t}\cdot g_{p^{2t}\cdot j}\\ &\hspace{1cm} - \leg{j}{p}\sum_{t = 0}^{n}\leg{j}{p}^{n-t} p^{t}\cdot g_{p^{2t}\cdot j}\\
&= p^{n+1}g_{p^{2n+2}\cdot j}\\
&\equiv 0 \pmod{p^{n+1}}.
\end{align*}

We then treat the cases where $D = p^{2v}j, p^2 \nmid j$. Then, by Proposition \ref{HeckegD}, considering a difference of Hecke operators, we have, again by \eqref{eqgD},
\allowdisplaybreaks{
\begin{align*}
 g_D | \Tp{2n + 2} - \leg{j}{p}g_D | \Tp{2n} &= g_j| \Tp{2n-2v+2} + \sum_{t = 1}^{v} p^{n-v+t+1}\cdot g_{p^{2n-2v+ 4t+2}\cdot j}\\
 &\hspace{0.5cm} - \leg{j}{p}g_j| \Tp{2n-2v} - \leg{j}{p}\sum_{t = 1}^{v} p^{n-v+t}\cdot g_{p^{2n -2v+ 4t}\cdot j}\\
&= p^{n -v+1}g_{p^{2n -2v+2}\cdot j}\\ &\hspace{1cm}+ p^{n-v+1} \sum_{t=0}^{v-1}\left(p^{t+1}g_{p^{2n-2v+4t+4}\cdot j} - \leg{j}{p}p^tg_{p^{2n-2v + 4t+2}\cdot j}\right)\\
&\equiv 0 \pmod{p^{n-v+1}}.
\end{align*}}
\end{proof}

\begin{Proposition}\label{GD} Let $G_D$ be one of Folsom-Ono's weakly holomorphic modular forms of weight $3/2$ on $\Gamma_0(144)$. We write $D = p^{2m}\cdot j$, where $m\geq 0$ and $p^2 \nmid j$. For $n$ sufficiently large:
\begin{align}
\label{eqGDmod} G_D | \Tps{12}{2n+2} - \leg{12}{p}\leg{j}{p}\cdot G_D | \Tps{12}{2n} \equiv 0 \pmod{p^{n-m+1}}.
\end{align}
\end{Proposition}
\begin{proof}
The proof of Proposition \ref{GD} follows exactly as that of Proposition \ref{gD}, with Proposition \ref{HeckeGD} taking the place of Proposition \ref{HeckegD} and a factor of $\leg{12}{p}$ along with each factor of $\leg{j}{p}$.
\end{proof}

\subsubsection{Proof of Theorem \ref{Hecke}}
By Zagier in \cite{ZAG}, the space of weakly holomorphic modular forms in plus space has an integral basis of the $g_D$'s considered in Proposition \ref{gD}. Thus, we can express $g$ as a sum of such $g_D$. Let
\[ g =\sum_D c_D \cdot g_D.\]
If we express $D = p^{2v_D} \cdot j_D$ where $p^2 \nmid j_D$, let $w = \max\{v_D\}$
Then, for $n \geq w$,
\allowdisplaybreaks{\begin{align*}
g | \Tp{2n+4} - g|\Tp{2n} &= \sum_D c_D \cdot \left(g_D | \Tp{2n+4} -g_D|\Tp{2n}\right)\\
&= \sum_D c_D \cdot \left(g_D | \Tp{2n+4} - \leg{j_D}{p}g_D | \Tp{2n+2}\right.\\
&\hspace{1.5 cm} +\left. \leg{j_D}{p}g_D | \Tp{2n+2} - g_D |\Tp{2n}\right)\\
&= \sum_D c_D \cdot \left(g_D | \Tp{2n+4} - \leg{j_D}{p}g_D | \Tp{2n+2} \right) \\
& \hspace{1.5 cm}+ \sum_{p \nmid j_D} c_D\leg{j_D}{p}\left(g_D | \Tp{2n+2} - \leg{j_D}{p}g_D |\Tp{2n}\right)\\
& \hspace{1.5 cm}-  \sum_{p | j_D} c_D\cdot g_D | \Tp{2n}.
\end{align*}}
Here, using Proposition \ref{gD}, the first sum is equivalent to $0\pmod{p^{n-w+2}}$, the second sum is equivalent to $0 \pmod{p^{n-w+1}}$, and the third sum is equivalent to $0 \pmod{p^{n-w}}$. Thus,
\begin{align*}
g | \Tp{2n+4} - g|\Tp{2n} &\equiv 0 \pmod{p^{n-w}}
\end{align*}

Given Proposition \ref{GD}, the proof of Theorem \ref{Hecke} (2) follows exactly as does that of Theorem \ref{Hecke} (1) after Proposition \ref{gD}, with the exception that a factor of $\leg{12}{p}$ appears along with each factor of $\leg{j}{p}$.

\subsection{Proof of Theorem \ref{Today}}
We first prove two Lemmas needed to prove Theorem \ref{Today}.
\begin{Lemma} \label{Thing1} For $n \geq v$, we have that
\begin{align}
\label{bpmod} b_{p^n}\left(p^{2v}j, d\right) -\leg{j}{p} b_{p^{n-1}}\left(p^{2v}j,d\right) &\equiv 0 \pmod{p^{n-v}}\\
\intertext{and}
\label{Bpmod} B_{p^n}\left(p^{2v}j, d\right) -\leg{12}{p}\leg{j}{p} B_{p^{n-1}}\left(p^{2v}j,d\right) &\equiv 0 \pmod{p^{n-v}}
\end{align}
 \end{Lemma}
 
\begin{proof}
Recalling that we let the coefficients of $g_D, G_D$ after the action of the Hecke operator be denoted as 
\begin{align*}  g_D | \Tp{2n} &= q^{-D} + \sum_d b_{p^n}(D,d)q^d\\
G_D | \Tps{12}{2n} &= q^{-D} + \sum_d B_{p^n}(D,d)q^d,
\end{align*}
we equate the coefficients of the $q^d$-th powers on either side of the equations \eqref{eqgDmod} and \eqref{eqGDmod} in Propositions \ref{gD} and \ref{GD}, respectively.
\end{proof}

\begin{Lemma} \label{Thing2} For $n \geq u$,
\begin{align}
\label{ap} a_{p^n}\left(D,p^{2u}i\right) &= \sum_{t = 0}^{n-u} \leg{-i}{p}^{n-u-t}p^ua\left(D,p^{2t}i\right)+ \sum_{t = 1}^{u}p^{u-t} a\left(D,p^{2n-2u +4t}\cdot i\right),\\
\intertext{and}
\label{Ap} A_{p^n}\left(D,p^{2u}i\right) &= \sum_{t = 0}^{n-u} \leg{12}{p}^{n-u-t}\leg{-i}{p}^{n-u-t}p^uA\left(D,p^{2t}i\right)+ \sum_{t = 1}^{u}p^{u-t} A\left(D,p^{2n-2u +4t}\cdot i\right).
\end{align}
\end{Lemma}

\begin{proof}
Recalling that we let the coefficients of $f_d$ after the action of the Hecke operator be denoted as 
\begin{align*}  f_d | \Tp{2n} &= \sum_D a_{p^n}(D,d)q^D\\
 F_d | \Tps{12}{2n} &= \sum_D A_{p^n}(D,d)q^D,
\end{align*}
we equate the coefficients of the $q^D$-th powers on either side of \eqref{eqfd} and \eqref{eqFd} in Propositions \ref{Heckefd} and  \ref{HeckeFd}, respectively.
\end{proof}

\begin{proof}[Proof of Theorem \ref{Today}]
If we let $u = 0$ in \eqref{ap}, we obtain
$$a_{p^n}(D,i) = \sum_{t = 0}^{n} \leg{-i}{p}^{n-t}b\left(D,p^{2t}i\right).$$
However, we also know, by \eqref{aandb} of Theorem \ref{AB} that
$$a_m(D,d) = -b_m(D,d)$$
for all $m,D,d$. Then, substituting into \eqref{bpmod} of Lemma \ref{Thing1},
\begin{align*}
b_{p^n}\left(p^{2v}j, i\right) - \leg{j}{p}b_{p^{n-1}}\left(p^{2v}j,i\right) &\equiv 0 \pmod{p^{n-v}}\\
-a_{p^n}\left(p^{2v}j, i\right) +\leg{j}{p} a_{p^{n-1}}\left(p^{2v}j,i\right)&\equiv 0 \pmod{p^{n-v}} \\
-\sum_{t = 0}^{n} \leg{-i}{p}^{n-t}a\left(p^{2v}j,p^{2t}i\right) + \leg{j}{p}\sum_{t = 0}^{n-1} \leg{-i}{p}^{n-t-1}a\left(p^{2v}j,p^{2t}i\right)&\equiv 0 \pmod{p^{n-v}}\\
\end{align*}
If $\leg{j}{p} = \leg{-i}{p}$, then the terms for $0 \leq t \leq n-1$ cancel between sums and
\begin{align*}
 -a\left(p^{2v}j,p^{2n}i\right) \equiv 0 \pmod{p^{n-v}}.
\end{align*}
So, applying Theorem \ref{AB} again, 
$$  b\left(p^{2v}j,p^{2n}i\right) \equiv 0 \pmod{p^{n-v}}.$$
In particular, when we then substitute $n = v+s$,
$$  b\left(p^{2v}j,p^{2v+ 2s}i\right) \equiv 0 \pmod{p^{s}}.$$
\end{proof}

As in previous sections, this proof follows exactly as that of Theorem \ref{Today} (1), with the exception that a factor of $\leg{12}{p}$ appears with every factor of $\leg{-i}{p}$.

\bibliographystyle{plain}
\bibliography{wackobill}{}

\end{document}